\documentclass[10pt,leqno]{amsart}
\usepackage{graphicx}
\usepackage{indentfirst,csquotes}
\usepackage[indentafter]{titlesec}
\titleformat{name=\section}{}{\thetitle.}{0.8em}{\centering\scshape}
\titleformat{name=\subsection}[runin]{}{\thetitle.}{0.5em}{\bfseries}[.]
\titleformat{name=\subsubsection}[runin]{}{\thetitle.}{0.5em}{\itshape}[.]
\titleformat{name=\paragraph,numberless}[runin]{}{}{0em}{}[.]
\titlespacing{\paragraph}{0em}{0em}{0.5em}
\titleformat{name=\subparagraph,numberless}[runin]{}{}{0em}{}[.]
\titlespacing{\subparagraph}{0em}{0em}{0.5em}
\usepackage{lipsum}

\usepackage{amssymb,amsthm,amsmath}
\usepackage{xcolor,paralist,hyperref,titlesec,fancyhdr,etoolbox,diagbox,multirow}
\usepackage[all]{xy}
\theoremstyle{definition}
\newtheorem{definition}{Definition}[]
\newtheorem{remark}[definition]{Remark}
\newtheorem{example}[definition]{Example}
\theoremstyle{plain}
\newtheorem{theorem}[definition]{Theorem}
\newtheorem{corollary}[definition]{Corollary}

\newtheorem{question}[definition]{Question}
\newtheorem{lemma}[definition]{Lemma}
\newtheorem{proposition}[definition]{Proposition}
\renewenvironment{proof}{\noindent\textsc{Proof.}\quad}{\qed}

\hypersetup{ colorlinks=true, linkcolor=black, filecolor=black, urlcolor=black }

\newcommand\A{\mathop{}\!\mathbb{A}}
\newcommand\Z{\mathop{}\!\mathbb{Z}}
\newcommand\Q{\mathop{}\!\mathbb{Q}}

\newcommand\p{\mathop{}\!\mathfrak{p}}
\newcommand\pp{\mathop{}\!\mathfrak{P}}

\newcommand\OK{\mathop{}\!\mathcal{O}}
\newcommand\Mat{\mathop{}\!\mathrm{Mat}}
\newcommand\Gal{\mathop{}\!\mathrm{Gal}}

\newcommand\Pic{\mathop{}\!\mathrm{Pic}}
\newcommand\Spec{\mathop{}\!\mathrm{Spec}}
\newcommand\Res{\mathop{}\!\mathrm{Res}}
\newcommand\GL{\mathop{}\!\mathrm{GL}}
\newcommand\GG{\mathop{}\!\mathbb{G}}

\newcommand\mo{\mathop{}\!\mathrm{mod}}

\newcommand\X{\mathop{}\!\mathbf{X}}

\newcommand\stab{\mathop{}\!\text{Stab}}
\begin{document}
\title{On bounds of extension degrees for similarity of integral matrices over number fields} 
\author{Mingqiang FENG and Ziyang ZHU}
\date{\today}
\address{School of Mathematical Sciences, Capital Normal University, Beijing 100048, China}
\email{wildmorph@outlook.com, zhuziyang@cnu.edu.cn}
\maketitle

\let\thefootnote\relax

\begin{abstract}
It is well-known that if $n\times n$ integral matrices $A$ and $B$ of a number field $K$ are similar over all completions of the ring of integers of $K$, then $A$ and $B$ are similar over the ring of integers of a finite extension of $K$. We prove that there is no uniform bound of the degree of extension of $K$ valid for all $n\times n$ matrices. On the other hand,  we provide a upper bound of the degree of extension of $K$ for a given separable characteristic polynomial.
\end{abstract}

\bigskip
\section{Introduction}\label{s1}

It is a classical problem to study the similarity of integral matrices of a number field $K$. In particular, Guralnick \cite[Theorem 7]{gu80} proved the following theorem

\begin{theorem}[Guralnick]
Let $M,N\in\Mat_n(\OK_K)$ where $K$ is a number field $K$ and $\OK_K$ is the ring of integers of $K$. If $M$ and $N$ are similar over each local completion $\OK_{K_{\p}}$ for all finite places, then there is a finite extension $L/K$ such that $M$ and $N$ are similar over $\OK_L$.
\end{theorem}

It should be pointed out that the proof given in \cite{gu80} is neither constructive nor effective. It is natural to ask whether there is a uniform bound of the degree of extension $L/K$ for all such matrices in $\Mat_n(\OK_K)$. More precisely, if $M$ and $N$ are similar over each local completion $\OK_{K_{\p}}$ for all finite places $\p$, we define
\[C_K(M,N):=\min\{[L:K]:L/K\text{ is a finite extension such that }M,N\text{ are similar over }\OK_L\}.\]

\begin{question}
Whether $C_K(M,N)$ can be bounded when $M$ and $N$ run over $\Mat_n(\OK_K)$ such that $M$ and $N$ are similar over each local completion $\OK_{K_{\p}}$ for all finite places $\p$?
 \end{question}

In this paper, we first give a negative answer to this question (see Theorem \ref{t5}).

\begin{theorem} \label{main1}
There is a sequence of matrices $M_n$ and $N_n$ in $\Mat_n(\Z)$ where $M_n$ and $N_n$ are similar over $\Z_p$ for all primes $p$ such that
\[\lim_{n\to\infty}C_{\Q}(M_n, N_n)=\infty.\]
\end{theorem}

On the other hand, we prove that there is an effective upper bound when the characteristic polynomial of $M$ and $N$ is a fixed separable polynomial over $K$ (see Theorem \ref{t13}).

\begin{theorem}
Let $\chi(t)$ be a monic separable polynomial over $\OK_K$, then there is an effective constant $C=C(f,n,K)$ depending only on $f,n$ and $K$ such that
\[\max\{C_K(M,N):\text{$M$ and $N$ are similar over $\OK_{K_{\p}}$ for all places $\p$ with $\mathrm{ch}(M)=\chi$}\}\leq C,\]
where $\mathrm{ch}(M)$ is the characteristic polynomial of $M$.
\end{theorem}

Notation and terminology are standard if not explained. For $n\times n$ matrices $M,N\in\Mat_n(R)$ over a commutative ring $R$, we say that $M$ and $N$ are similar over $R$ if there exists $U\in\GL_n(R)$ such that $UM=NU$. Let $K$ be a number field and $\OK_K$ be the ring of integers of $K$. For any finite place $\p$ of $K$, the completion of $K$ with respect to $\p$ is denoted by $K_{\p}$ and $\OK_{K_{\p}}$ is the integral ring of $K_{\p}$. We write $\OK_{K_{\p}}:=K_{\p}$ if $\p$ is infinite.

The paper is organized as follows. In \S \ref{s2}, we provide a lower bound for $C_K(M,N)$ and show this lower bound can be arbitrarily large by explicit constructing certain required matrices over $\mathbb Z$. In \S \ref{s4}, we give an upper bound for $C_K(M,N)$ if the characteristic polynomial is fixed and separable. The proof of this result relies on the work of Wei and Xu \cite{wx12, wx13} and some techniques from class field theory.

\subsection*{Acknowledgments}
The authors would like to thank Professor Fei XU for bringing this problem to their attention.

\section{Establishing the Lower Bound}\label{s2}

Let us assume that $M,N\in\Mat_n(\OK_K)$ are non-derogatory and for each finite place $\p$ of $K$ there is a $U_{\p}\in\GL_n(\OK_{K_{\p}})$ such that $U_{\p}M=NU_{\p}$. Define the integral centralizer of $M$ by
\[R_{M;K}:=\{A\in\Mat_n(\OK_K):AM=MA\},\]
note that $\OK_K[M]\subseteq R_{M;K}\subseteq K[M]$, so $R_{M;K}$ is commutative. Consider the $R_{M;K}$-module
\[T_{M,N;K}:=\{A\in\Mat_n(\OK_K):AM=NA\},\]
whose $R_{M;K}$-module structure is given by $r(A):=Ar$ for $r\in R_{M;K}$ and $A\in T_{M,N;K}$. We claim such a module is a line bundle on $\Spec(R_{M;K})$.

\begin{lemma}\label{t1}
$T_{M,N;K}$ is a locally free $R_{M;K}$-module of rank one.
\end{lemma}
\begin{proof}
Fix a finite place $\p$ of $K$. Since there is $U_{\p}\in\GL_n(\OK_{K_{\p}})$ such that $U_{\p}M=NU_{\p}$, we claim $(T_{M,N;K})_{\p}=U_{\p}(R_{M;K})_{\p}$ as $(R_{M;K})_{\p}$-modules. Indeed, for $A\in(T_{M,N;K})_{\p}$ we have $U_{\p}^{-1}AM=U_{\p}^{-1}NA=MU_{\p}^{-1}A$, so $U_{\p}^{-1}A\in(R_{M;K})_{\p}$, which means $A\in U_{\p}(R_{M;K})_{\p}$. Conversely, for $r\in(R_{M;K})_{\p}$, we have $(U_{\p}r)M=U_{\p}rM=U_{\p}Mr=NU_{\p}r$, hence $U_{\p}r\in(T_{M,N;K})_{\p}$. This proves the claim.
\end{proof}

Thus, $T_{M,N;K}$ defines a class $[T_{M,N;K}]$ in $\Pic(R_{M;K})$, the group of isomorphism classes of locally free rank one $R_{M;K}$-modules, with the group law is given by the tensor product. Note that $\Pic(R_{M;K})$ is a finite abelian group, because $R_{M;K}$ is an $\OK_K$-finite order.

\begin{lemma}[Faddeev, {\cite[Theorem 6]{gu80}}]\label{t2}
Assume that $M$ and $N$ are locally similar over $\OK_{K_{\p}}$ for all finite places $\p$. Then $M,N$ are similar over $\OK_K$ if and only if $[T_{M,N;K}]$ is the trivial class in $\Pic(R_{M;K})$.
\end{lemma}

Hence, $[T_{M,N;K}]$ measures the failure of the local integral similarities to glue to a global integral similarity. The following lemma describes the behavior of this class under base-change.

\begin{lemma}\label{t3}
Let $L$ be a finite extension of $K$ with ring of integers $\OK_L$, we have
\begin{itemize}
\item $R_{M;K}\otimes_{\OK_K}\OK_L=\{A\in\Mat_n(\OK_L):AM=MA\}$.
\item $T_{M,N;K}\otimes_{\OK_K}\OK_L=\{A\in\Mat_n(\OK_L):AM=NA\}$.
\end{itemize}
\end{lemma}
\begin{proof}
Consider the $\OK_K$-linear map
\[\Phi_M:\Mat_n(\OK_K)\longrightarrow\Mat_n(\OK_K),\quad A\longmapsto AM-MA,\]
one can verify $R_{M;K}=\ker(\Phi_M)$. Note that $\OK_L$ is flat over $\OK_K$, we have
\[R_{M;K}\otimes_{\OK_K}\OK_L=\ker(\Phi_M)\otimes_{\OK_K}\OK_L=\ker(\Phi_M\otimes_{\OK_K}\OK_L)=\{A\in\Mat_n(\OK_L):AM=MA\}.\]
The proof for the second item is the same: we only need to consider
\[\Psi_{M,N}:\Mat_n(\OK_K)\longrightarrow\Mat_n(\OK_K),\quad A\longmapsto AM-NA,\]
then $T_{M,N;K}=\ker(\Psi_{M,N})$. By flatness, one obtains $T_{M,N;K}\otimes_{\OK_K}\OK_L=\ker(\Psi_{M,N}\otimes_{\OK_K}\OK_L)=\{A\in\Mat_n(\OK_L):AM=NA\}$.
\end{proof}

By Lemma \ref{t3}, there is a finite flat morphism $i:\Spec(R_{M;L})\to\Spec(R_{M;K})$ of degree $[L:K]$, and one has $i^*([T_{M,N;K}])=[T_{M,N;L}]$.

\begin{theorem}\label{t4}
Let $K$ be a number field, and let $M,N\in\Mat_n(\OK_K)$ be non-derogatory matrices. If $M$ and $N$ are similar over $\OK_{K_{\p}}$ for each finite place $\p$ of $K$, then
\[C_K(M,N)\geq\mathrm{ord}_{\Pic(R_{M;K})}([T_{M,N;K}]).\]
\end{theorem}
\begin{proof}
Suppose $L/K$ is a finite extension such that $M$ and $N$ are similar over $\OK_L$, that is, $UM=NU$ for some $U\in\GL_n(\OK_L)$. For $A\in T_{M,N;L}$ we have $U^{-1}A\in R_{M;L}$, so $T_{M,N;L}=UR_{M;L}$ as $R_{M;L}$-modules. This is obviously a rank one free module, so $i_*([T_{M,N;L}])=0$. By the projection formula \cite[Proposition 21.5.6]{gd67},
\[0=i_*i^*([T_{M,N;K}])=[L:K]\cdot[T_{M,N;K}]\in\Pic(R_{M;K}),\]
thus the order of $[T_{M,N;K}]$ in $\Pic(R_{M;K})$ divides $[L:K]$. Taking the minimum over all finite extensions $L/K$ for which $M,N$ are similar over $\OK_L$ gives the desired lower bound.
\end{proof}

Based on the above estimation, we give a proof of Theorem \ref{main1} by explicit construction.

\begin{theorem}\label{t5}
For any positive integer $C_0$, there exist two matrices $M_0,N_0\in\Mat_n(\Z)$ which are similar over each $\Z_p$, such that $C_{\Q}(M_0,N_0)\geq C_0$.
\end{theorem}
\begin{proof}
We first focus on the special case of $n=2$. Consider two matrices over $\Z$,
\[M_0'=\left(\begin{array}{cc}1&b\\0&1+q\end{array}\right)\quad\text{and}\quad N_0'=\left(\begin{array}{cc}0&1\\-(1+q)&2+q\end{array}\right),\]
where $q$ is a odd prime in $\Z$ such that $q\geq2C_0+1$, $b\in\Z$ is a primitive root modulo $q$.

First, we claim $M_0'$ and $N_0'$ are similar over each $\Z_p$, where $p\in\Z$ is a prime. If $p=q$, note that $\gcd(q,b)=1$, one can verify $$U=\left(\begin{array}{cc}1&0\\1&b\end{array}\right)\in\GL_2(\Z_q)$$ satisfies $UM_0'=N_0'U$. If $p\neq q$, one can also verify $$V=\left(\begin{array}{cc}1&(1-b)q^{-1}\\1&(1-b+q)q^{-1}\end{array}\right)\in\GL_2(\Z_p)$$ satisfies $VM_0'=N_0'V$.

Next, we compute $\Pic(R_{M_0';\Q})$. By definition, our $R_{M_0';\Q}$ can be identified with
\[\left\{\left(\begin{array}{cc}x&by\\0&x+qy\end{array}\right):x,y\in\Z\right\}\overset{\sim}{\longrightarrow}\{(u,v)\in\Z^{\oplus2}:u\equiv v(\mo~q)\},\quad\left(\begin{array}{cc}x&by\\0&x+qy\end{array}\right)\longmapsto(x,x+qy),\]
that is,
\[R_{M_0';\Q}=\Z I_2+\Z(M_0'-I_2)=\Z[M_0']\cong\Z[t]/(t-1)(t-1-q),\]
whose normalization is $\Z^{\oplus2}$ and the conductor is $\mathfrak{f}=(q\Z)^{\oplus2}$. Consider the Milnor conductor square
\[\xymatrix{R_{M_0';\Q}\ar[d]\ar[r]&\Z^{\oplus2}\ar[d]\\R_{M_0';\Q}/\mathfrak{f}\ar[r]&\Z^{\oplus2}/\mathfrak{f}}\]
where one can compute $R_{M_0';\Q}/\mathfrak{f}\cong\mathbb{F}_q$ and $\Z^{\oplus2}/\mathfrak{f}\cong\mathbb{F}_q^{\oplus2}$ directly. By the Units-Pic sequence \cite[Theorem I.3.10]{we13} we obtain
\[\Pic(R_{M_0';\Q})\cong(\mathbb{F}_q^{\times}\times\mathbb{F}_q^{\times})/(\Delta(\mathbb{F}_q^{\times})\cdot(\{\pm1\}\times\{\pm1\}))\cong\mathbb{F}_q^{\times}/\{\pm1\}.\]

Finally, we determine $[T_{M_0',N_0';\Q}]$. Let $$W=\left(\begin{array}{cc}(1-b+q)q^{-1}&(b-1)q^{-1}\\-1&1\end{array}\right)\in\GL_2(\Q)$$ which is chosen so that $WT_{M_0',N_0';\Q}$ is identified with a fractional $R_{M_0';\Q}$-module inside $\Q[M_0']\cong\Q^{\oplus2}$:
\[WT_{M_0',N_0';\Q}=\left\{\left(\begin{array}{cc}x&y+b(b-1)q^{-1}x\\0&bx+qy\end{array}\right):x,y\in\Z\right\}\overset{\sim}{\longrightarrow}\{(u,v)\in\Z^{\oplus2}:v\equiv bu(\mo~q)\}.\]
So $WT_{M_0',N_0';\Q}$ is a rank one locally free $R_{M_0';\Q}$-module, where $W$ is principal. The order of $[T_{M_0',N_0';\Q}]$ in $\Pic(R_{M_0';\Q})$ is equal to
\[\mathrm{ord}_{\mathbb{F}_q^{\times}/\{\pm1\}}([b])=\frac{q-1}{2},\]
because the image of $[WT_{M_0',N_0';\Q}]$ in $\mathbb{F}_q^{\times}/\{\pm1\}$ is $[b]$, which is primitive, under a suitable choice of the identification. Thus, by Theorem \ref{t4}, $C_{\Q}(M_0',N_0')\geq\frac{q-1}{2}\geq C_0$.

For the general case $n\geq2$, let
\[M_0=\left(\begin{array}{cc}M_0'&\\&2I_{n-2}\end{array}\right)\quad\text{and}\quad N_0=\left(\begin{array}{cc}N_0'&\\&2I_{n-2}\end{array}\right).\]
For a number field $L/\Q$, if $M_0$ and $N_0$ are similar over $\OK_L$, it is easy to check $M_0'$ and $N_0'$ are similar over $\OK_L$. Hence $C_{\Q}(M_0,N_0)\geq C_{\Q}(M_0',N_0')\geq \frac{q-1}{2}\geq C_0$.
\end{proof}

\begin{remark}\label{t6}
Let $M_0'$ and $N_0'$ be the matrices defined in the proof of Theorem \ref{t5}, one can show $C_{\Q}(M_0',N_0')\leq\frac{q-1}{2}$, see Example \ref{t15} below. Therefore, we have
\[C_{\Q}(M_0',N_0')=\frac{q-1}{2},\]
which means our lower and upper bounds are optimal in a certain sense.
\end{remark}

\section{The Upper Bounds}\label{s4}
Although Theorem \ref{t5} claims that $C_K(\cdot,\cdot)$ admits no uniform upper bound, $C_K(\cdot,\cdot)$ does possess a uniform upper bound (Theorem \ref{t13}) when the characteristic polynomial is fixed and assumed to be separable. To determine this bound, we employ a criterion for the existence of integral points (Corollary \ref{t9}); the conditions imposed by this criterion on the id\`{e}le class group allow us to construct the required field extensions via the class field theory. As an application, we conclude this section with Example \ref{t15}, which completes the discussion on the upper bound initiated in Remark \ref{t6}.
\subsection{Existence of Integral Points}\label{s4.1}
Let $K$ be a number field with ring of integers $\OK_K$, denote $\Omega_K$ the set of all places of $K$. For $\p\in\Omega_K$ a finite place, the completion of $K$ (resp. $\OK_K$) at $\p$ is denoted as $K_{\p}$ (resp. $\OK_{K_{\p}}$). We write $\OK_{K_{\p}}:=K_{\p}$ if $\p\in\Omega_K$ is infinite. The ad\`{e}le ring of $K$ is denoted as $\A_K$, whose units form a group $\A_K^{\times}$ and we call it the id\`{e}le group of $K$.

Let $M,N\in\Mat_n(\OK_K)$ be non-derogatory matrices, define a closed subscheme $\X_{M,N;K}$ of $\GL_n$ over $\OK_K$ as
\[\X_{M,N;K}:=\{A\in\GL_n:AM=NA\}.\]
Suppose $M,N$ are similar over each $\OK_{K_{\p}}$ for $\p$ a place of $K$, that is,
\[\prod_{\p\in\Omega_K}\X_{M,N;K}(\OK_{K_{\p}})\neq\emptyset.\]
In particular, $M,N$ are similar over $K$. Therefore, the generic fiber $X_{M,N;K}:=\X_{M,N;K}\times_{\OK_K}K$ has a $K$-point $P$. For any $K$-algebra $R$, the map $A\mapsto P^{-1}A$ induces a bijection
\begin{align*}
X_{M,N;K}(R)&=\{A\in\GL_n(R):AM=NA\}\\
&=\{A\in\GL_n(R):AM=PMP^{-1}A\}\\
&=\{A\in\GL_n(R):P^{-1}AM=MP^{-1}A\}\\
&\overset{\sim}{\to}\{B\in\GL_n(R):BM=MB\}.
\end{align*}
Let $\chi(t)$ denote the characteristic polynomial of $M$, note that the map
\[(R[t]/\chi(t))^{\times}\overset{\sim}{\longrightarrow}X_{M,N;K}(R),\quad g(t)\longmapsto Pg(M)\]
is a bijection, because any matrix that commutes with $M$ must be a polynomial in $M$. Hence, if we denote $G:=\Res_{(K[t]/\chi(t))/K}(\GG_m)$, the action
\[X_{M,N;K}\times_KG\longrightarrow X_{M,N;K},\quad(B,g)\longmapsto Bg(M)\]
defines a $G$-torsor over $K$.

In particular, if we assume $\chi(t)=\prod_{i=1}^n(t-a_i)$ splits into distinct roots in $K$ (indeed, in $\OK_K$), then the above bijection becomes $(R^{\times})^{\oplus n}\overset{\sim}{\to}X_{M,N;K}(R)$. In this situation, $G=\GG_m^n$ over $K$ and $X_{M,N;K}$ is a trivial $\GG_m^n$-torsor over $K$.

We refer to \cite{wx12} and \cite{wx13} to introduce a criterion for the existence of integer points on $\X_{M,N;K}$. To this end, we shall henceforth assume that $\chi(t)$ is split and separable, so $G=\GG_m^n$ over $K$.

Since $X_{M,N;K}$ is a trivial $G$-torsor over $K$, one can fix an isomorphism $X_{M,N;K}\cong G$ induced by a $K$-point $P$. Note that $G(\A_K)=(\A_K^{\times})^{\oplus n}$, we define the map
\[f_{M,N;K}:\prod_{\p\in\Omega_K}\X_{M,N;K}(\OK_{K_{\p}})\hookrightarrow X_{M,N;K}(\A_K)\cong G(\A_K)=(\A_K^{\times})^{\oplus n}.\]
This map is well-defined because $\X_{M,N;K}$ is separated over $\OK_K$.

Consider the following open subgroup of $(\A_K^{\times})^{\oplus n}$ \cite[Lemma 1.2]{wx12}:
\[\stab_{\A}(\X_{M,N;K}):=(\A_K^{\times})^{\oplus n}\cap\left(\prod_{\p\in\Omega_K}\stab(\X_{M,N;K}(\OK_{K_{\p}}))\right),\]
where
\[\stab(\X_{M,N;K}(\OK_{K_{\p}})):=\left\{g\in(K_{\p}^{\times})^{\oplus n}:g\X_{M,N;K}(\OK_{K_{\p}})=\X_{M,N;K}(\OK_{K_{\p}})\right\}\]
for each finite place $\p$ of $K$ and the action
\[g:X_{M,N;K}(K_{\p})\longrightarrow X_{M,N;K}(K_{\p}),\quad B\longmapsto Bg(M)\]
is continuous. Here we can view $\X_{M,N;K}(\OK_{K_{\p}})$ as an open subset of $X_{M,N;K}(K_{\p})$, since $\X_{M,N;K}$ is separated over $\OK_K$. Then, $\X_{M,N;K}(\OK_K)\neq\emptyset$ if and only if
\[\mathrm{im}(f_{M,N;K})\cap\left((K^{\times})^{\oplus n}\cdot\stab_{\A}(\X_{M,N;K})\right)\neq\emptyset,\]
by \cite[Proposition 1.4]{wx12}. Indeed, this result can be generalized to a more general case immediately.

\begin{proposition}\label{t8}
Let $\Xi$ be an open subgroup of $\stab_{\A}(\X_{M,N;K})$, then $\X_{M,N;K}(\OK_K)\neq\emptyset$ if and only if
\[\mathrm{im}(f_{M,N;K})\cap\left((K^{\times})^{\oplus n}\cdot\Xi\right)\neq\emptyset.\]
\end{proposition}

Since $\Xi$ is open in $(\A_K^{\times})^{\oplus n}$, we can find $\{\Xi_i\}_{i=1}^n$ such that each $\Xi_i$ is an open subgroup in $\A_K^{\times}$ and $\prod_{i=1}^n\Xi_i\subseteq\Xi$. By the class field theory, there is a finite abelian extension $K_i/K$ such that the Artin map
\[\phi_{K_i/K}:\A_K^{\times}/K^{\times}\Xi_i\overset{\sim}{\longrightarrow}\Gal(K_i/K)\]
gives the isomorphism for $1\leq i\leq n$. Hence, Proposition \ref{t8} can be restated as follows.

\begin{corollary}[{\cite[Corollary 1.6]{wx12}}]\label{t9}
Let $\{\Xi_i\}_{i=1}^n$ be as above, denote $f_{M,N;K}^i$ the $i$-th component of $f_{M,N;K}$. Then, $\X_{M,N;K}(\OK_K)\neq\emptyset$ if and only if there is $x\in\prod_{\p\in\Omega_K}\X_{M,N;K}(\OK_{K_{\p}})$ such that $\phi_{K_i/K}(f_{M,N;K}^i(x))$ is trivial in $\Gal(K_i/K)$, for all $1\leq i\leq n$.
\end{corollary}

\subsection{An Upper Bound}\label{s4.2}
In order to apply Proposition \ref{t8} and Corollary \ref{t9}, we shall compute $\stab_{\A}(\X_{M,N;K})$. This leads us to compute $\stab(\X_{M,N;K}(\OK_{K_{\p}}))$. By definition,
\[\stab(\X_{M,N;K}(\OK_{K_{\p}})):=\left\{g\in(K_{\p}^{\times})^{\oplus n}:g\X_{M,N;K}(\OK_{K_{\p}})=\X_{M,N;K}(\OK_{K_{\p}})\right\}.\]
\begin{lemma}\label{t10}
Let $g\in(K_{\p}^{\times})^{\oplus n}$, then $g\in\stab(\X_{M,N;K}(\OK_{K_{\p}}))$ if and only if $g(M)\in\GL_n(\OK_{K_{\p}})$.
\end{lemma}
\begin{proof}
This can be verified directly from the definition. Note that $Ag(M)M=AMg(M)=NAg(M)$ for all $A\in\X_{M,N;K}(\OK_{K_{\p}})$, which means $Ag(M)$ always belongs to $\X_{M,N;K}(\OK_{K_{\p}})$.
\end{proof}

Hence, our first step is to characterize $g(M)\in\Mat_n(\OK_{K_{\p}})$ that are invertible, where $g\in(K_{\p}^{\times})^{\oplus n}$. In fact, determining the invertibility of such $g(M)$ is not straightforward. While one typically desires $M\in\Mat_n(\OK_K)$ to be triangulable, this property is generally guaranteed only over principal ideal domains \cite[Theorem III.12]{ne72}. However, by passing to a sufficiently large class field, we can obtain an analogous result (see Lemma \ref{t11}). Consequently, we restrict our attention to the invertibility of $g(M)$ when $K$ is replaced by a sufficiently large class field (see Proposition \ref{t12}).
\subsubsection{On the Invertibility of $g(M)$}\label{s4.2.1}
Let $K$ be a number field, there is a tower of fields
\[K=H_0\subseteq H_1\subseteq\cdots\subseteq H_{i-1}\subseteq H_i\subseteq\cdots,\]
where $H_i$ is the Hilbert class field of $H_{i-1}$, $i\geq1$.

\begin{lemma}\label{t11}
Let $M\in\Mat_n(\OK_K)$ with characteristic polynomial $\chi(t)=\prod_{i=1}^n(t-a_i)$. Then $M$ is similar to an upper triangular matrix in $\GL_n(\OK_{H_{n-1}})$.
\end{lemma}
\begin{proof}
Let $\mathbf{v}$ be a non-zero column vector in $\OK_K^n$ such that $M\mathbf{v}=a_1\mathbf{v}$. We claim there is $U\in\GL_n(\OK_{H_1})$ such that $\mathbf{v}=U(s,0,\cdots,0)^T$ for some non-zero $s\in\OK_{H_1}$. So $U^{-1}MU(s,0\cdots,0)^T=a_1(s,0\cdots,0)^T$. Thus, the first column of $U^{-1}MU$ is $(a_1,0,\cdots,0)^T$, which means $M$ is similar to
\[\left(\begin{array}{cccc}a_1&*&\cdots&*\\0&*&\cdots&*\\\cdots&\cdots&\cdots&\cdots\\0&*&\cdots&*\end{array}\right)\]
in $\GL_n(\OK_{H_1})$. It then suffices to proceed by induction on the remaining $(n-1)\times(n-1)$ submatrix. Note that each inductive step requires passing to a larger Hilbert class field.

It only needs to prove the claim. We start with the case $n=2$. Suppose $\mathbf{v}=(v_1,v_2)^T$, where $v_1,v_2\neq0$. Consider the ideal $(v_1,v_2)\subseteq\OK_K$, we have $(s)=(v_1,v_2)\OK_{H_1}$ is a principal ideal of $\OK_{H_1}$, since $H_1$ is the Hilbert class field of $K$. Now, there exist $r_1,r_2\in\OK_{H_1}$ such that $v_1=r_1s$, $v_2=r_2s$ and $(r_1,r_2)=\OK_{H_1}$, so there is an invertible matrix $U=\left(\begin{array}{cc}r_1&*\\r_2&*\\\end{array}\right)$, such that $(v_1,v_2)^T=U(s,0)^T$. For the case $n>2$, it suffices to proceed by induction on the number of non-zero entries in the column vector $\mathbf{v}$.
\end{proof}

Fix a matrix $M\in\Mat_n(\OK_K)$ with characteristic polynomial $\chi(t)=\prod_{i=1}^n(t-a_i)$. By Lemma \ref{t11}, $M$ is similar to
\[\left(\begin{array}{cccc}a_1&*&\cdots&*\\0&a_2&\cdots&*\\\cdots&\cdots&\cdots&\cdots\\0&0&\cdots&a_n\end{array}\right)\]
in $\GL_n(\OK_{H_{n-1}})$. Let $g=(g_1,\cdots,g_n)\in(H_{n-1,\pp}^{\times})^{\oplus n}$, where $\pp$ is a finite place of $H_{n-1}$. Without loss of generality, one can assume $g(M)$ has form
\[\left(\begin{array}{cccc}g(a_1)&b_{12}&\cdots&b_{1n}\\0&g(a_2)&\cdots&b_{2n}\\\cdots&\cdots&\cdots&\cdots\\0&0&\cdots&g(a_n)\end{array}\right),\quad b_{kl}\in H_{n-1,\pp},\]
if we regard $g=g(t)$ as the image of $g=(g_1,\cdots,g_n)$ under the bijection
\[(H_{n-1,\pp}^{\times})^{\oplus n}\overset{\sim}{\to}(H_{n-1,\pp}[t]/\chi(t))^{\times},\]
by the Chinese remainder theorem. It is easy to verify that $g(a_i)=g_i$, $1\leq i\leq n$. Hence, $g(M)\in\GL_n(\OK_{H_{n-1,\pp}})$ if and only if $b_{kl}\in\OK_{H_{n-1,\pp}}$ and $g_i\in\OK_{H_{n-1,\pp}}^{\times}$. Determining a necessary and sufficient condition for $b_{kl}\in\OK_{H_{n-1,\pp}}$ is hard; for our purposes, it suffices to derive a necessary condition, this requires a more precise description of $g(t)$.

To express $g(t)\in(H_{n-1,\pp}[t]/\chi(t))^{\times}$ explicitly in terms of its values $g(a_i)=g_i\neq0$ at $a_i$, $1\leq i\leq n$, we apply the Lagrange interpolation formula:
\[g(t)=\sum_{i=1}^n\left(g_i\prod_{j=1,j\neq i}^n\frac{t-a_j}{a_i-a_j}\right)=\sum_{r=0}^{n-1}\left(\sum_{i=1}^nL_{ri}g_i\right)t^r,\]
where
\[L_{ri}:=\frac{(-1)^{n-r-1}}{\prod_{j=1,j\neq i}^n(a_i-a_j)}E_{n-r-1}\Big(\{a_1,\cdots,a_n\}\setminus\{a_i\}\Big)\in K\]
and $E_m$ is the $m$-th elementary symmetric polynomial of the elements in the given set.

\begin{proposition}\label{t12}
Let $M\in\Mat_n(\OK_K)$ with characteristic polynomial $\chi(t)=\prod_{i=1}^n(t-a_i)$, denote
\[\delta_{\chi}:=\prod_{1\leq i<j\leq n}(a_i-a_j),\quad a_i\in\OK_K.\]
Suppose $g=(g_1,\cdots,g_n)\in(H_{n-1,\pp}^{\times})^{\oplus n}$. If $g_i\in\OK_{H_{n-1,\pp}}^{\times}$ and $v_{\pp}(g_i-g_n)\geq v_{\pp}(\delta_{\chi})$ for all $1\leq i\leq n$, then $g(M)\in\GL_n(\OK_{H_{n-1,\pp}})$ and hence, $g\in\stab(\X_{M,N;K}(\OK_{H_{n-1,\pp}}))$.
\end{proposition}
\begin{proof}
When choose $g=(1,\cdots,1)$, we obtain $g(t)\equiv1$ is a constant and $\{L_{ri}\}$ admits relations
\[L_{0n}=1-\sum_{i=1}^{n-1}L_{0i};\quad L_{rn}=-\sum_{i=1}^{n-1}L_{ri},~(1\leq r\leq n-1).\]
Now, for general $g=(g_1,\cdots,g_n)$, we have
\[\sum_{i=1}^nL_{0i}g_i=\sum_{i=1}^{n-1}L_{0i}g_i+L_{0n}g_n=\sum_{i=1}^{n-1}L_{0i}(g_i-g_n)+g_n\]
and
\[\sum_{i=1}^nL_{ri}g_i=\sum_{i=1}^{n-1}L_{ri}g_i+L_{rn}g_n=\sum_{i=1}^{n-1}L_{ri}(g_i-g_n),~(1\leq r\leq n-1).\]
Hence, the coefficients of $g(t)$ are integral in $H_{n-1,\pp}$ if each $v_{\pp}(g_i-g_n)\geq v_{\pp}(\delta_{\chi})$. This means each $b_{kl}$ is integral because it is a linear combination of elements in $\OK_{H_{n-1}}$ with coefficients in $\OK_{H_{n-1,\pp}}$.
\end{proof}

\subsubsection{Establishing the Upper Bound}\label{s4.2.2}
With Proposition \ref{t12} in hand, one can apply Corollary \ref{t9} to show, if the characteristic polynomial is fixed, then $C_K(M,N)$ admits a uniform upper bound which depends on $K$ and the fixed characteristic polynomial.
\begin{theorem}\label{t13}
Let $K$ be a number field, suppose $M\in\Mat_n(\OK_K)$ with characteristic polynomial $\chi(t)=\prod_{i=1}^n(t-a_i)$. If $N\in\Mat_n(\OK_K)$ is similar to $M$ over each $\OK_{K_{\p}}$, then
\[C_K(M,N)\leq[H_{n-1}:K]\cdot[\A_{H_{n-1}}^{\times}:H_{n-1}^{\times}\Xi_0],\]
where
\[\Xi_0:=\prod_{\pp|\delta_{\chi}}\left(1+\pp^{v_{\pp}(\delta_{\chi})}\right)\times\prod_{\pp\nmid\delta_{\chi}}\OK_{H_{n-1,\pp}}^{\times}.\]
\end{theorem}
\begin{proof}
If we assume $g_i\in1+\pp^{v_{\pp}(\delta_{\chi})}$ for all $1\leq i\leq n$, then $g\in\stab(\X_{M,N;H_{n-1}}(\OK_{H_{n-1,\pp}}))$, by Proposition \ref{t12}. Therefore, one can find an open subgroup $\Xi:=(\Xi_0)^{\oplus n}\subseteq\stab_{\A}(\X_{M,N;H_{n-1}})$. By the generalized principal ideal theorem for id\`{e}le class groups \cite[Chapter VI.7]{ne10}, there is a finite extension $F/H_{n-1}$ of degree $[\A_{H_{n-1}}^{\times}:H_{n-1}^{\times}\Xi_0]$ and a natural extension $\overline{\Xi}_0$ of $\Xi_0$ such that the map
\[\A_{H_{n-1}}^{\times}/H_{n-1}^{\times}\Xi_0\to\A_F^{\times}/F^{\times}\overline{\Xi}_0\]
is trivial. Note that for $L/K$ a finite extension and $\mathfrak{v}$ lies above $\p$, the matrices $M,N$ are similar over $\OK_{K_{\p}}$ implies they are similar over $\OK_{L_{\mathfrak{v}}}$, so
\[\emptyset\neq\prod_{\p\in\Omega_K}\X_{M,N;K}(\OK_{K_{\p}})\hookrightarrow\prod_{\mathfrak{v}\in\Omega_L}\X_{M,N;L}(\OK_{L_{\mathfrak{v}}}).\]
Hence, we obtain a commutative diagram
\[\xymatrix{\prod_{\pp}\X_{M,N;H_{n-1}}(\OK_{H_{n-1,\pp}})\ar^{f_{M,N;H_{n-1}}^i}[rrr]\ar@{^(->}[d]&&&\A_{H_{n-1}}^{\times}/H_{n-1}^{\times}\Xi_0\ar_{0}[d]\ar[rr]^{\phi_{F/H_{n-1}}}&&\Gal(F/H_{n-1})\ar[d]\\
\prod_{\mathfrak{Q}}\X_{M,N;F}(\OK_{F_{\mathfrak{Q}}})\ar^{f_{M,N;F}^i}[rrr]&&&\A_F^{\times}/F^{\times}\overline{\Xi}_0\ar[rr]^{\phi_{E/F}}&&\Gal(E/F)}\]
for some finite extension $E/F$. It is clear that $\phi_{E/F}\circ f_{M,N;F}^i\left(\prod_{\pp}\X_{M,N;H_{n-1}}(\OK_{H_{n-1,\pp}})\right)$ is trivial. By Corollary \ref{t9}, $\X_{M,N;F}$ has an $\OK_F$-point.
\end{proof}

In general, if $\chi(t)$ does not split over $K$, Theorem \ref{t13} remains applicable upon passing to the splitting field of $\chi(t)$. In this case, the established upper bound must be adjusted by a factor of $n!$, representing the maximum possible order of the Galois group and the maximal degree of the splitting extension. Furthermore, the construction involving the Hilbert class field must be transitioned accordingly to the splitting field.

This theorem gives an effective version of \cite[Theorem 7]{gu80} (which said, two matrices $M,N\in\Mat_n(\OK_K)$ are similar over all the local rings of $\OK_K$ if and only if $M,N$ are similar over some finite integral extension of $\OK_K$) in the non-derogatory semi-simple case, and partially generalizes \cite[Proposition 3.5]{hl25}:

\begin{corollary}\label{t14}
Let $\OK_K$ be a principal ideal domain and let $M\in\Mat_n(\OK_K)$ be a diagonal matrix with pairwise distinct diagonal entries. Then, $N\in\Mat_n(\OK_K)$ is similar to $M$ over each $\OK_{K_{\p}}$ if and only if $M,N$ are similar over $\OK_K$.
\end{corollary}
\begin{proof}
Indeed, $[H_{n-1}:K]=1$ because $\OK_K$ is principal. Furthermore, all $b_{kl}=0$ and hence are integral. So $g(M)\in\GL_n(\OK_{K_{\p}})$ if and only if $g_i\in\OK_{K_{\p}}^{\times}$, which does not depend on $\delta_{\chi}$. Thus, $C_K(M,N)=1$, since one can take $\Xi=\left(\prod_{\p}\OK_{K_{\p}}^{\times}\right)^{\oplus n}$ to be the trivial one and the class number of $K$ is equal to $1=\left[\A_K^{\times}:K^{\times}\prod_{\p}\OK_{K_{\p}}^{\times}\right]$.
\end{proof}

\begin{example}\label{t15}
Let $M_0'$ and $N_0'$ be the matrices over $\Z$ defined in the proof of Theorem \ref{t5}. Since $\Z$ is a principal ideal domain, by Theorem \ref{t13} we have
\[C_{\Q}(M_0',N_0')\leq\left[\A_{\Q}^{\times}:\Q^{\times}\left((1+q\Z_q)\times\prod_{p\neq q}\Z_p^{\times}\times\mathbb{R}^{\times}\right)\right]=\frac{[\Z_q^{\times}:1+q\Z_q]}{\#\{\pm1\}}=\frac{q-1}{2}.\]
The upper and lower bounds coincide in this case.
\end{example}


\begin{thebibliography}{00}

\bibitem[GD67]{gd67} A. Grothendieck, J. Dieudonn\'{e}. \textit{\'{E}l\'{e}ments de G\'{e}om\'{e}trie Alg\'{e}brique IV: \'{E}tude Locale des Sch\'{e}mas et des Morphismes de Sch\'{e}mas, Quatri\`{e}me Partie}, Publications Math\'{e}matiques de L'I.H.\'{E}.S., {\bf 32} (1967), 5-361.
\bibitem[Gu80]{gu80} R. M. Guralnick. \textit{A Note on the Local-Global Principle for Similarity of Matrices}, Linear Algebra and its Applications {\bf 30} (1980), 241-245.
\bibitem[HL25]{hl25} K. Huang, Y. Liu. \textit{Local-global Principle for Triangularizability and Diagonalizability of Matrices}, arXiv:2511.15827.
\bibitem[Neu10]{ne10} J. Neukirch. \textit{Algebraic Number Theory}, Springer-Verlag (2010).
\bibitem[New72]{ne72} M. Newman. \textit{Integral Matrices}, Academic Press (1972).
\bibitem[We13]{we13} C. A. Weibel. \textit{The K-Book: An Introduction to Algebraic K-Theory}, American Mathematical Society (2013).
\bibitem[WX12]{wx12} D. Wei, F. Xu. \textit{Integral Points for Multi-norm Tori}, Proceedings of the American Mathematical Society, (5) {\bf 104} (2012), 1019-1044.
\bibitem[WX13]{wx13} D. Wei, F. Xu. \textit{Integral Points for Groups of Multiplicative Type}, Advances in Mathematics, {\bf 232} (2013), 36-56.
\end{thebibliography}
\end{document}